\documentclass[conference]{IEEEtran}

\usepackage{amsmath,amssymb,amsthm,amsfonts,amsopn,cite,mathrsfs,todonotes,mathtools}
\usepackage[T1]{fontenc}
\usepackage{ifpdf}
\usepackage[english]{babel}
\usepackage[ruled]{algorithm2e}
\usepackage{dsfont,url}
\usepackage{tikz}
\usepackage{pgfplots}
\pgfplotsset{compat=1.14}
\usepackage{xcolor}
\usepackage{booktabs}

\usepackage{multirow}
\usetikzlibrary{patterns}
\usetikzlibrary{calc}
\usetikzlibrary{arrows.meta}

\makeatletter

\usepackage{bm}
\usepackage{hyperref}

\newcommand{\bd}{\mathbf}

\newcommand{\LtR}{\mathbf L^2(\mathbb{R})}
\newcommand{\FLtRp}{\mathcal F^{-1}(\mathbf L^2(\mathbb{R}^+))}

\newcommand{\RR}{\mathbb{R}}
\newcommand{\ZZ}{\mathbb{Z}}
\newcommand{\NN}{\mathbb{N}}
\newcommand{\CC}{\mathbb{C}}

\newcommand{\abs}[1]{|#1|}

\newcommand{\wlet}{\psi}
\newcommand{\Wlet}{\Psi}
\newcommand{\Hil}{\mathcal{H}}

\newcommand{\oscV}{{\mathrm{osc}}_{\mathcal{V}}}

\theoremstyle{plain}
\newtheorem{theorem}{Theorem}[section]
\newtheorem{proposition}[theorem]{Proposition}
\newtheorem{lemma}[theorem]{Lemma}
\newtheorem{corollary}[theorem]{Corollary}
\theoremstyle{definition}

\theoremstyle{remark}
\theoremstyle{remark}
\newtheorem*{rem*}{Remark}

\usepackage{matlab-prettifier}
\usepackage[T1]{fontenc}
\lstMakeShortInline"
\begin{document}
\author{\IEEEauthorblockN{Nicki Holighaus\IEEEauthorrefmark{1} and
G\"unther Koliander\IEEEauthorrefmark{2}}
\IEEEauthorblockA{Acoustics Research Institute, Austrian Academy of Sciences, Wohllebengasse 12--14, 1040 Vienna, Austria.\\
Email: \IEEEauthorrefmark{1}nicki.holighaus@oeaw.ac.at,
\IEEEauthorrefmark{2}guenther.koliander@oeaw.ac.at}}

 \title{Rotated time-frequency lattices are sets of stable sampling for continuous wavelet systems}
\maketitle
\begin{abstract}%
We provide an example for the generating matrix $A$ of a two-dimensional lattice $\Gamma = A\ZZ^2$, such that the following holds: For any sufficiently smooth and localized mother wavelet $\psi$, there is a constant $\beta(A,\psi)>0$, such that $\beta\Gamma\cap (\RR\times\RR^+)$ is a set of stable sampling for the wavelet system generated by $\psi$, for all $0<\beta\leq \beta(A,\psi)$. The result and choice of the generating matrix are loosely inspired by the studies of low discrepancy sequences and uniform distribution modulo $1$. In particular, we estimate the number of lattice points contained in any axis parallel rectangle of fixed area. This estimate is combined with a recent sampling result for continuous wavelet systems, obtained via the oscillation method of general coorbit theory. 
\end{abstract}


\IEEEpeerreviewmaketitle
\section{Introduction}

A wavelet system~\cite{da92} is a collection of functions generated from a single prototype, the \emph{mother wavelet}, by translation and dilation. Here, we consider a complex mother wavelet $\wlet\in\LtR$, such that its Fourier transform vanishes for negative frequencies, i.e., $\hat{\wlet}(\xi) = 0$ for $\xi\in (-\infty,0]$. With a slight abuse of notation, we denote the space of square-integrable functions with this property as  $\FLtRp$. It is a Hilbert space with respect to the standard inner product on $\bd L^2(\RR)$. Mother wavelets $\wlet\in \FLtRp$ are sometimes called \emph{analytic}, although the terminology \emph{analytic wavelet transform} is not used consistently~\cite{ltfatnote053,Lilly2010analytic}. 
The associated continuous wavelet system is 
defined as 
\begin{equation}\label{eq:WletSystem}
   \Wlet = (\wlet_{x,s})_{x\in\RR,s\in\RR^+},\ \text{where}
\end{equation}
\begin{equation}\label{eq:WletAtoms}
   \wlet_{x,s} := \bd T_x \bd D_{1/s} \wlet = \sqrt{s} \wlet\left(s(\bullet-x)\right).
\end{equation}
Here, $\bd D_s$ with $s\in\RR^+$ denotes a unitary dilation by the factor $s$, and $\bd T_x$ denotes translation by $x\in\RR$. We note that the above definition is slightly unusual in the sense that we associate large scales with small values of $s$~\cite{da92}. Under the given conditions on the mother wavelet $\wlet$, the natural extension of $\Wlet$ to negative scales $s<0$ has no further consequences on our results and is therefore omitted. In the following, we refer to $\Lambda = \RR\times\RR^+$ as the \emph{phase space} associated with $\Wlet$.

Sets of stable sampling for the continuous wavelet system $\Wlet$ are discrete sets $\Lambda_D\subset \Lambda$, such that $(\wlet_{x,s})_{(x\, ;\, s)^{\intercal}\in \Lambda_D}$ forms a frame, cf.\ Section \ref{ssec:wavelets}. Such sets have been systematically studied at least since the popularization of frame theory and multiresolution analyses in the 1980s~\cite{dagrme86,mallat1989multiresolution,meyer1993}, although certain wavelet bases were known earlier, as discussed in~\cite[Chapter 4.2.1]{da92}. Most established sampling schemes for wavelets follow the same formula: 
\begin{enumerate}
  \item Select a basis $a\in (1,\infty)$ and consider all scales $a^j$, $j\in\ZZ$.
  \item Select a relative translation step $b\in (0,\infty)$ and consider, at scale $a^j$, the translations $a^{-j}  \cdot lb$, $l\in\ZZ$. 
\end{enumerate}
The density of the resulting set $\Lambda_D =\big\{ 
\binom{a^{-j} \cdot lb}{a^j}
\colon j,l\in\ZZ\big\}\subset \Lambda$ can be controlled by scaling $a$ and $b$. Although variations on this scheme are occasionally studied, e.g., for the construction of shift-invariant wavelet systems~\cite{RON1997408}, they often only differ by adding additional points to $\Lambda_D$.

A few recent works have explored low discrepancy sets and sequences, as commonly used in the quasi-Monte Carlo method~\cite{dick_pillichshammer_2010}, for selecting discrete subsets from $\Psi$ and other time-frequency systems~\cite{levie2023quasi,holighaus2023grid}. In particular, the numerical results in 
the work \cite{holighaus2023grid} suggest that the restriction of certain sheared lattices to the upper half-plane form sets of stable sampling for \emph{inhomogeneous} wavelet systems which are obtained by removing the scales $s<1$ and adding an appropriate substitute. 

\vspace{6pt}
\noindent\textit{Contribution: } In this work, we consider the rotated square lattice  $\Gamma = A\ZZ^2$, with $A = \begin{psmallmatrix} 1 & -\alpha \\ \alpha & 1 \end{psmallmatrix}$, for a specific choice of $\alpha$. We show that isotropic dilations of $\Gamma$, restricted to the upper half-plane, form sets of stable sampling for the continuous wavelet system $\Wlet$. The main step towards this result is achieved by proving that the number of lattice points contained in an arbitrary half-open rectangle is proportional to its area, provided that the rectangle has at least a certain minimal area. Our result 
relies on choosing $\alpha$ to be badly approximable, a property that also plays a prominent role in 
the construction of Kronecker sequences~\cite{kuinie,DT97}. The choice of $\Gamma$ further bears some resemblance to lattice rules~\cite{Hickernell1998}, a popular construction rule for low discrepancy sequences. 

\section{Preliminaries}
\label{ssec:wavelets}

For any $\wlet\in\FLtRp$, the continuous wavelet transform defined for all $(x\, ;\, s)^{\intercal}\in\Lambda$ by 
\begin{equation}\label{eq:waveletttransform}
    \begin{split}
    W_{\wlet} f (x,s) = \langle f, \wlet_{x,s} \rangle_{L_2},\ \ \text{for all } f\in\FLtRp, 
    \end{split}
\end{equation}
is a function in $\bd L^\infty(\Lambda)$. If the so-called admissibility constant $C_\wlet = \|\wlet/(\bullet)\|^2_2$ is finite, then the system $\Wlet\subset \FLtRp$ is a \emph{continuous tight frame}\cite{antoin2,Christensen2016} with frame bound $C_\wlet$. In particular, if $C_\wlet = 1$, then 
$\|f\|_{\bd L^2(\RR)}=\|W_{\wlet} f\|_{\bd L^2(\Lambda)}$, for all $f\in \FLtRp$. The image space $\widetilde{\Hil} = W_{\wlet}(\FLtRp)$ of the wavelet transform with mother wavelet $\wlet$ is a reproducing kernel Hilbert space of continuous, square-integrable functions~\cite{da92,Christensen2016}.

This work is concerned with the study of certain discrete subsets $\Lambda_D\subset \Lambda$, such that 
 $(\wlet_{x,s})_{(x\, ;\, s)^{\intercal}\in\Lambda_D}$ 
is a discrete frame~\cite{Christensen2016}, i.e., 
\begin{equation}\label{eq:frameineq}
  A\|f\|_2^2 \leq \sum_{(x\, ;\, s)^{\intercal}\in\Lambda_D} |\langle f,\wlet_{x,s}\rangle|^2 \leq B\|f\|_2^2,
\end{equation}
for all $f\in\FLtRp$ and some constants $0<A\leq B<\infty$. When we say that $\Lambda_D$ is a set of stable sampling for the continuous wavelet system $\Wlet$, we mean precisely that $\Lambda_D$ is a set of stable sampling for $\widetilde{\Hil}$, which, noting $C_\wlet \|f\|_2^2 = \|W_{\wlet} f\|_{\widetilde{\Hil}}^2$, is equivalent to \eqref{eq:frameineq}.


\section{The Frame Property from Oscillation Estimates}\label{ssec:oscill}

The oscillation method was introduced by Feichtinger and Gr\"ochenig in their seminal works studying atomic decompositions in \emph{coorbit spaces}~\cite{feichtinger1988unified,feichtinger1989banach,feichtinger1989banach2}. Although the discretization results in those works can be applied to the wavelet transform on $\FLtRp$, it is nontrivial to show the density and separation requirements therein for the specific point sets $c\Gamma$ that we consider. Hence, it is more convenient to rely on \emph{generalized coorbit theory}, as introducted in~\cite{fora05} and extended in \cite{rauhut2011generalized,kempka2015general}. Applied to continuous frames on Hilbert spaces $\Hil$, this variant of the oscillation method yields a sufficient condition for irregular sets of stable sampling that can informally be summarized as follows: Assume that $\Phi = (\phi_{m})_{m\in M}$ is a localized~\cite{fora05}, continuous tight frame for $\Hil$, and there is a countable covering $\mathcal{V} = (V_j)_{j\in J}$ of $M$, such that the oscillation of $\Phi$ with respect to $\mathcal{V}$ is small in the norm of some Schur-type algebra $\mathcal A$~\cite{VoHoModules}. Then any set $(m_j)_{j\in J}$, with $m_j\in V_j$ for all $j\in J$, is a set of stable sampling for $\Phi$, or equivalently, $(\phi_{m_j})_{j\in J}$ is a discrete frame for $\Hil$. More precisely, given a function $\Xi\colon \Lambda\times\Lambda \rightarrow \CC$, with $|\Xi|\equiv 1$, the $\Xi$-oscillation $\oscV$ of $\Phi$ with respect to $\mathcal{V}$ is given by 
\begin{equation}
    \begin{split}
    \lefteqn{\oscV(m_0,m_1)}\\
    & = \sup_{n \in \mathcal V_{m_1}}
         \left|
            \langle \phi_{m_0},\phi_{m_1}\rangle - \Xi(m_1,n)\langle \phi_{m_0},\phi_{n}\rangle \right|,
    \end{split}
         \label{eq:def_genosckern}        
    \end{equation}
where $\mathcal V_{m_1} := \bigcup_{j \in J: m_1 \in V_j} V_j$. The appropriate choice of $\Xi$ is crucial to achieve small $\mathcal A$-norm of $\oscV$.

In a recent paper concerned with warped time-frequency systems~\cite{bahowi15}, of which the continuous wavelet systems $\Psi$ are a special case, it was shown that there exists a family of coverings $\mathcal{V^\delta} = (V^{\delta}_j)_{j\in J}$ of $\Lambda = \RR\times\RR^+$, comprised of half-open rectangles of identical area $\delta^2$, such that the oscillation of the wavelet system $\Psi$ with respect to $\mathcal{V}^\delta$ converges to zero in $\mathcal A$-norm, for $\delta\rightarrow 0$. In particular, \cite[Example 4.1, Corollary 6.9]{bahowi15} imply the following result.

\begin{theorem}\label{thm:SmallOscYieldsFrame}
  Let $\wlet\in\FLtRp$, with $C_\wlet=1$, such that $\mathcal F(\psi)  \in \mathcal C^2(\RR)$, with 
  $\max\{(\bullet)^5,(\bullet)^{-5}\}\cdot \mathcal F(\psi)  \in \mathcal C_0(\RR)\cap \bd L^2(\RR^+)$. Define $\mathcal V^\delta = (V^\delta_{k,\ell})_{k,\ell\in\ZZ}$ by
  \[
   V^\delta_{k,\ell} = \left[\frac{\delta^2 k}{|I^\delta_\ell|},\frac{\delta^2 (k+1)}{|I^\delta_\ell|}\right) \times I^\delta_\ell,
  \]
  with $I^\delta_\ell = [e^{\delta \ell},e^{\delta (\ell+1)})$. Then there exists $\delta(\Wlet) > 0$, such that for any $0<\delta \leq \delta(\Wlet)$, and any discrete set $\Lambda_D\subset \Lambda$ with
  \[\inf_{k,\ell\in\ZZ} |\Lambda_D \cap V^\delta_{k,\ell}| > 0 \text{ and } \sup_{k,\ell\in\ZZ} |\Lambda_D \cap V^\delta_{k,\ell}| =:N < \infty,\]
  the collection $(\wlet_{x,s})_{(x\, ;\, s)^{\intercal}\in\Lambda_D}$ is a frame for $\FLtRp$.
\end{theorem}

Since Theorem \ref{thm:SmallOscYieldsFrame} is an application of the results referred to above, its proof is not self-contained and deferred to the Appendix. In the next section, we show that certain rotated lattices satisfy the conditions on $\Lambda_D$ in Theorem \ref{thm:SmallOscYieldsFrame}.

\section{Wavelet Frames by Sampling on Lattices}\label{sec:lattices}

In the following, we set $\alpha = \varphi^{-1} = \frac{\sqrt{5}-1}{2}$, where $\varphi$ is the golden ratio, and consider the matrix
\begin{equation}
    A = \begin{pmatrix}1 & -\alpha \\ \alpha & 1 \end{pmatrix}\; .     
\end{equation}
With this choice of $A$, we define the lattice
\begin{equation}
    \Gamma = A \ZZ^2,
\end{equation}
i.e., $\Gamma = \{ A \begin{psmallmatrix}n \\m\end{psmallmatrix} \colon n,m\in\ZZ\}$. As we will see, any axis parallel rectangle $\Delta$ of a given size has clear bounds on the number of points in $\Delta \cap \Gamma$ that do not depend on its position or ratio of its side lengths. Before we proceed to prove a formal version of this statement, we derive some properties of $\alpha$ that will be subsequently used.

\begin{proposition}\label{pro:AlphaProperties}
  Let $\alpha = \varphi^{-1} = \frac{\sqrt{5}-1}{2}$. Then, for all $n\in\NN$, 
  \begin{equation}\label{eq:AlphaLinearRelation}
    \alpha^n = (-1)^{n-1}(f_n\alpha - f_{n-1}),
  \end{equation}
  where $f_n$ is the $n$-th Fibonacci number, defined by $f_0 = 0$, $f_1 = 1$ and, for $n\geq 2$, $f_n = f_{n-1}+f_{n-2}$.

  \medskip{}

  Moreover, for any $n\in\NN$, we have 
  \begin{equation}\label{eq:AlphaPowersFibonacciConstant} 
    \alpha^{n-1}f_{n+2} + \alpha^n f_{n+1} = 2+\alpha = \varphi^2.
  \end{equation}
\end{proposition}
\begin{proof}
 We first prove \eqref{eq:AlphaLinearRelation}. For $n=1$, the statement trivially holds, while for $n=2$,
  \begin{equation}\label{eq:alphasq}
    \alpha^2 = \varphi^{-2} = \frac{1}{1+\varphi} = 1 - \frac{\varphi}{1+\varphi} = 1 - \alpha 
  \end{equation}
  as desired. 
  The induction step follows from
  \[
    \begin{split}
    \alpha^{n+1} 
    = \alpha^n\cdot \alpha 
    & = (-1)^{n-1}(f_n\alpha - f_{n-1})\cdot \alpha\\
    & \stackrel{\hidewidth \eqref{eq:alphasq} \hidewidth}= (-1)^{n-1} (f_n (1-\alpha) - f_{n-1}\alpha)\\
    & = (-1)^{n-1} (f_n - (f_n + f_{n-1})\alpha) \\
    & = (-1)^n (f_{n+1}\alpha - f_n).
    \end{split}
  \]

  \medskip{}

  We next prove  \eqref{eq:AlphaPowersFibonacciConstant}. Inserting $n=1$ on the left hand side immediately yields
  \[
    \alpha^0 f_3 + \alpha^1 f_2 = 2 + \alpha, 
  \]
  as desired. 
  The induction step follows from
  \begin{align*}
    \alpha^{n}f_{n+3} + \alpha^{n+1} f_{n+2}
    & = \alpha^{n}(f_{n+2}+f_{n+1}) + \alpha^{n+1} f_{n+2} \\
    & = (\alpha^2+\alpha) \alpha^{n-1}f_{n+2} + \alpha^{n}f_{n+1} \\
    &  \stackrel{\hidewidth \eqref{eq:alphasq} \hidewidth}= \alpha^{n-1}f_{n+2} + \alpha^{n}f_{n+1} \\
    & = 2 + \alpha.\hfill\qedhere
  \end{align*}  
  
\end{proof}

We are now ready to show that any axis-parallel rectangle with sufficient area contains at least one lattice point. 

\begin{lemma}\label{lem:RectContainsLatticeElement}
    Let $\Delta = [a,b) \times [c,d)$ be an axis parallel half-open rectangle of size $\mu(\Delta)\geq 2+\alpha = \varphi^2$, where $\mu$ is the standard Lebesgue measure on $\Lambda = \RR\times \RR^+$. 
    Then $\lvert \Delta \cap \Gamma \rvert \geq 1$.
\end{lemma}
\begin{proof}
  Given $\Delta = [a,b) \times [c,d)$, with $\Delta\cap \Gamma = \emptyset$, we assume, without loss of generality, that there exist $a_0 \in [a,b)$ and $c_0 \in [c,d)$, such that 
  \[
  \binom{b}{c_0},\, \binom{a_0}{d}\in \Gamma.
  \]
  Otherwise, there is a nonzero $\begin{psmallmatrix} b_0 \\d_0\end{psmallmatrix}\in [0,\infty)^2$, such that $(\Delta\ +\ [0,b_0)\times [0,d_0))\,\cap\, \Gamma = \emptyset$.

  We now show that $\Delta\cap \Gamma = \emptyset$ implies that $\mu(\Delta)<2+\alpha$. To this end, we first consider the case that both $b-a\geq 1$ and $d-c\geq 1$. 
  Because $\binom{a_0+\alpha}{d-1}\in \Gamma$ and $d-1\in [c,d)$, we have $a_0+\alpha\geq b$, since otherwise $\binom{a_0+\alpha}{d-1}\in \Gamma\cap \Delta$. 
  Similarly, because $\binom{a_0-1}{d-\alpha}\in \Gamma$ and $d-\alpha\in [c,d)$,
  we have $a_0 - 1 < a$.
  Hence, 
  \begin{equation*}
      b \leq a_0+\alpha < a +1 + \alpha
  \end{equation*}
  and we conclude that $b-a < 1+\alpha$. 
  Analogously, 
  we obtain $d-c< 1+\alpha$.
  Together, this implies $\mu(\Delta) < (1+\alpha)^2$. Finally, \eqref{eq:alphasq} yields
  \[
   (1+\alpha)^2 = 1 + 2\alpha + \alpha^2 = 2+\alpha,
  \]
  as desired. 

  \medskip{}
  
  For the remainder of the proof, we assume that $d-c<1$ and remark that the proof for the case $b-a<1$ is analogous, using the rotational symmetry $\Gamma = \begin{psmallmatrix} 0 & -1 \\ 1 & 0 \end{psmallmatrix}\Gamma$.
  Furthermore, since $\Gamma\subset \RR^2$ is a lattice, it is in particular a subgroup of $\RR^2$: The problem is invariant under shifts $\Delta-\gamma$ of $\Delta$ by lattice points $\gamma\in \Gamma$. Hence, we can assume, without loss of generality, that $\begin{psmallmatrix} a_0\\ d\end{psmallmatrix} = \begin{psmallmatrix} 0\\ 0\end{psmallmatrix}$, implying $a\leq 0$, $b>0$ and $c<0$. Further, let $n\in\NN$ be such that 
  \[
    \alpha^n \leq -c < \alpha^{n-1}.
  \]
  In particular, this implies 
  \begin{equation}\label{eq:cleq0}
      c \leq -\alpha^n < -\alpha^{n+1} < 0.
  \end{equation}
  If $n\in 2\NN$ then, by \eqref{eq:AlphaLinearRelation} in Proposition \ref{pro:AlphaProperties}, the chain of inequalities in \eqref{eq:cleq0}  is equivalent to
  \begin{equation}\label{eq:cAndAlphanBounds}
    c\leq   f_n\alpha -f_{n-1} < f_n - f_{n+1}\alpha  <0.
  \end{equation}
  For $p_0 = \begin{psmallmatrix} f_{n}\\ -f_{n-1}\end{psmallmatrix}$ and $p_1 = \begin{psmallmatrix} -f_{n+1}\\ f_n\end{psmallmatrix}$, we have
  \[
    Ap_0 = \begin{pmatrix} f_n + \alpha f_{n-1} \\ f_n\alpha - f_{n-1}\end{pmatrix}\quad\text{and}\quad Ap_1 = \begin{pmatrix} -f_{n+1} - \alpha f_{n} \\ -f_{n+1}\alpha + f_{n}\end{pmatrix}.
  \]
  In particular, \eqref{eq:cAndAlphanBounds} yields $Ap_0\in ( (0,\infty)\times [c,0) )\,\cap\, \Gamma$ and $Ap_1\in ((-\infty,0)\times [c,0))\, \cap\, \Gamma$. 
  Thus, $\Delta\cap\Gamma = \emptyset$ implies $b \leq  (Ap_0)_1$ and $a > (Ap_1)_1$.
  Together, we obtain 
  \begin{align*}
    b-a 
    & <
    (Ap_0)_1 - (Ap_1)_1 \\
    & = f_n + \alpha f_{n-1} - (-f_{n+1} - \alpha f_n) 
    \\
    & = f_{n+2} + \alpha f_{n+1}.
  \end{align*}
  Similarly, if $n\in 2\NN-1$, we set $p_0 = \begin{psmallmatrix} -f_{n}\\f_{n-1}\end{psmallmatrix}$ and $p_1 = \begin{psmallmatrix} f_{n+1}\\{-f_{n}}\end{psmallmatrix}$ and we arrive at the same bound for $b-a$. 

  Since $d-c = -c < \alpha^{n-1}$, we obtain 
  \[
    \mu(\Delta) <  \alpha^{n-1}f_{n+2} + \alpha^n f_{n+1} = 2+\alpha,
  \]
  by \eqref{eq:AlphaPowersFibonacciConstant} in Proposition \ref{pro:AlphaProperties}. This completes the proof.
\end{proof}

Furthermore, any axis-parallel rectangle that has sufficiently small area contains no more than a single element of $\Gamma$.

\begin{lemma}\label{lem:SmallRectContainsAtMostOneLatticeElement}
     Let $\Delta = [a,b) \times [c,d)$ be an axis parallel half-open rectangle of size $\mu(\Delta)\leq 1/(3+2\alpha)$. 
    Then $\lvert \Delta \cap \Gamma \rvert \leq 1$.
\end{lemma}
\begin{proof}
  Given $\Delta = [a,b) \times [c,d)$, with $\Delta\cap \Gamma \neq \emptyset$, we assume, without loss of generality, that there exist $a_0 \in [a,b)$ and $c_0 \in [c,d)$, such that 
  \[
  \binom{a}{c_0},\, \binom{a_0}{c}\in \Gamma.
  \]
  Otherwise, there is a nonzero $\begin{psmallmatrix} b_0 \\d_0\end{psmallmatrix}\in [0,\infty)^2$, such that $[a+b_0,b)\times [c+d_0,d)\cap \Gamma \neq \emptyset$.
  We now show that $\lvert \Delta \cap \Gamma \rvert \geq 2$ implies $\mu(\Delta)> 1/(3+2\alpha)$.

  Analogous to the previous proof, we only consider the case $d-c<1$ and
  assume 
  that $\begin{psmallmatrix} a_0\\ c\end{psmallmatrix} = \begin{psmallmatrix} 0\\ 0\end{psmallmatrix}$, implying $a\leq 0$, $b>0$ and $d>0$.
  Since $\alpha$ is a so-called badly approximable number, we can bound how well it can be approximated by any rational. 
  Specifically, by \cite[Sec.~11.7]{hardy1979introduction}, we have 
    \begin{equation} \label{eq:fibonbound}
        \bigg\lvert \alpha + \frac{m}{n} \bigg\rvert \geq \frac{1}{(3+2\alpha)n^2}
    \end{equation}
    for any $n,m\in \ZZ$, $n\neq 0$.
    In turn, we have 
    $\lvert n \alpha + m \rvert \geq \frac{1}{(3+2\alpha)\abs{n}}$.
    Hence, any lattice point in $\Gamma$ with second component in $[0,d)$ must satisfy $\frac{1}{(3+2\alpha)\abs{n}} < d$, i.e., 
    \begin{equation}\label{eq:absnbound1}
        \abs{n} > \frac{1}{(3+2\alpha)d}.
    \end{equation}
    
    On the other hand, $\lvert n \alpha + m \rvert < d < 1$, implies that the cases $n>0$ and $m>0$ as well as $n<0$ and $m<0$ can be excluded, which implies $nm\leq 0$.
    Thus, for the first component of any point in $\Delta \cap \Gamma$, we have 
    $\lvert n - m \alpha \rvert = \lvert n \rvert+ \lvert m \alpha \rvert$.
    Furthermore, $a \leq n -  m \alpha < b$ and, hence,
    \begin{equation}\label{eq:absnbound2}
         \lvert n \rvert \leq \lvert n \rvert+ \lvert m \alpha \rvert < b-a.
    \end{equation}
    Combining \eqref{eq:absnbound1} and \eqref{eq:absnbound2}, we obtain 
    $\frac{1}{(3+2\alpha)} < (b-a) d = \mu(\Delta)$.\hfill\qedhere


\end{proof}

\begin{corollary}\label{cor:ProporitonalityAndSampling}
  Let $\Delta = [a,b) \times [c,d)$ be an axis parallel half-open rectangle of size $\mu(\Delta)= 2+\alpha = \varphi^2$. Then $1\leq\lvert \Delta \cap \Gamma \rvert \leq 12$.
  \medskip{}

  In particular, for any $\delta>0$, there is a $\beta := \beta(\delta)>0$, such that the following holds: If $\mathcal V^\delta = (V^\delta_{k,\ell})_{k,\ell\in\ZZ}$ is as in Theorem \ref{thm:SmallOscYieldsFrame}, then 
  \[
    1\leq \lvert \beta\Gamma\cap V^\delta_{k,\ell}\rvert \leq 12, 
  \]
  for all $k,\ell\in\ZZ$. Hence, if $\delta \leq \delta(\Wlet)$, we can choose $\Lambda_D = \beta\Gamma$ in Theorem \ref{thm:SmallOscYieldsFrame}.
\end{corollary}
\begin{proof}
  The lower bound in the first assertion follows from Lemma \ref{lem:RectContainsLatticeElement}. For the upper bound note that $2+\alpha < \tfrac{12}{2\alpha + 3}$, such that an axis-parallel, half-open rectangle of area $2+\alpha$ can be contained in no more than $12$ axis-parallel, half-open rectangles of area $(2\alpha + 3)^{-1}$. The bound now follows from Lemma \ref{lem:SmallRectContainsAtMostOneLatticeElement}.

  For the second assertion, it is sufficient to note that each $V^\delta_{k,l}$ is an axis-parallel, half-open rectangle of area $\mu(V^\delta_{k,l})=\delta^2$. Choose $\beta(\delta) = \delta^2\cdot (2+\alpha)^{-1}$ to obtain the desired result. 
\end{proof}

Considering the coarse estimates used in both previous proofs, it seems quite likely that the upper bound in Corollary \ref{cor:ProporitonalityAndSampling} can be further improved. 

\section{Conclusion and Outlook}

We have shown that a certain scale of rotated lattices provides sets of stable sampling for continuous wavelet systems, when restricted to the upper half plane. Our proof relies on prior work on discretization of wavelet systems in the context of coorbit spaces, 
and a property of the proposed lattices that evokes discrepancy theory. It should be noted that our result 
generalizes to other warped time frequency systems or, more generally, localized continuous frames that satisfy oscillation estimates with respect to a phase-space covering comprised of axis-parallel rectangles with identical area. A generalization to higher dimensional phase space seems quite feasible. 

We expect analogues of Lemmas \ref{lem:RectContainsLatticeElement} and \ref{lem:SmallRectContainsAtMostOneLatticeElement} for any badly approximable number in place of $\alpha$. Specifically, some bound in the style of Lemma \ref{lem:SmallRectContainsAtMostOneLatticeElement} exists for any badly approximable number, but \cite[Sec.~11.7]{hardy1979introduction} only yields explicit estimates for algebraic numbers. We are currently working on a proof for Lemma \ref{lem:RectContainsLatticeElement} that does not rely on the properties of the golden ratio $\varphi$ and generalizes to arbitrary badly approximable numbers. With this extended result, it will be possible to show that a more general class of time-frequency lattices generates sets of stable sampling for wavelet systems and other localized continuous frames.

\section*{Acknowledgments} 
We would like to 
thank Friedrich Pillichshammer for continued, inspiring discussions on the use and construction of low discrepancy sequences. 
G. K. gratefully acknowledges support by the Austrian Science Fund (FWF) project Y 1199.

\appendix
\begin{proof}[Proof of Theorem \ref{thm:SmallOscYieldsFrame}]
  First note that \cite[Example 4.1]{bahowi15} shows that $\Wlet$ is a warped time-frequency system with respect to the warping function $\Phi = \log$ (the natural logarithm) and the prototype $\theta = \mathcal F (\wlet) \circ \exp$, where $\circ$ denotes composition. In particular, it is a warping function in the sense of \cite[Definition 4.2]{bahowi15}, with associated weight function $w = (\Phi^{-1})' = \exp$. Clearly, $w\in \mathcal C^{\infty}(\RR)$ satisfies $w(t+s)=w(t)w(s)$, for all $t,s\in\RR$, such that it is self-moderate. Finally, the $k$-th derivative $w^{(k)}$ of $w$ equals $w$, such that $|w^{(k)}/w|=1$, and $\Phi=\log$ satisfies all assumptions on $\Phi$ in \cite[Corollary 6.9]{bahowi15}. The covering $\mathcal V^\delta$ in the statement of Theorem \ref{thm:SmallOscYieldsFrame} is precisely the $\Phi$-induced $\delta$-cover, for $\Phi=\log$, considered in \cite[Corollary 6.9]{bahowi15}. Further, Equation (16) in \cite{bahowi15} is satisfied with $Y=\bd L^2(\Lambda)$ and $m\equiv 1$, as stated in \cite[Section 3]{fora05}. It only remains to verify the conditions on $\theta = \mathcal F (\wlet) \circ \exp$, before \cite[Corollary 6.9]{bahowi15} can be applied. 
  
  With this choice of $m$, and estimating $w(t)=\exp(t)\leq \exp(|t|)$, the conditions on $\theta$ in \cite[Corollary 6.9]{bahowi15} simplify to 
  \begin{enumerate}
    \item $\theta\in\mathcal C^2(\RR)$ with $\theta^{(k)}\cdot \exp(|\bullet|)^3\in \mathcal C_0(\RR)$, for $0\leq k\leq 2$, 
    \item $\theta^{(k)} \cdot \exp(|\bullet|)^{\frac{7-2k}{2}}\in \bd L^2(\RR)$.
  \end{enumerate} 
  Since $\exp(|\log(\tau)|) = \max\{\tau,\tau^{-1}\}$, for all $\tau\in\RR^+$, the first condition is equivalent to 
  \[
    \theta^{(k)}(\log(\bullet))\cdot \max\{\bullet^3,\bullet^{-3}\}\in \mathcal C_0(\RR^+).
  \]
  Furthermore, 
  \begin{equation}\label{eq:ThetaLogFirstDerivative}
    \theta^{(1)}\circ \log(\tau) = \tau\cdot (\theta\circ \log)^{(1)}(\tau) = \tau\cdot (\mathcal F(\wlet))^{(1)}(\tau)
  \end{equation}
  \begin{equation}\label{eq:ThetaLogSecondDerivative}
    \text{and }\ \theta^{(2)}\circ \log(\tau) = \tau\cdot (\mathcal F(\wlet))^{(2)}(\tau) - (\mathcal F(\wlet))^{(1)}(\tau), 
  \end{equation}
  for all $\tau\in\RR^+$. We conclude that Item (1) is equivalent to $\max\{\bullet^4,\bullet^{-4}\}\cdot(\mathcal F(\wlet))^{(k)}\in \mathcal C_0(\RR)$, for all $0\leq k\leq 2$.

  Further, for all $l\in\NN$ and measurable $g\colon \RR \rightarrow \CC$,
  \[
    \begin{split}
  \lefteqn{\int_\RR |g(t)\cdot \exp(|t|)^l|^2 ~dt}\\
   & = \int_{\RR^+} \tau^{-1}|\max\{\tau^l,\tau^{-l}\}\cdot (g\circ \log) (\tau)|^2 ~d\tau.
    \end{split}
  \]
  Inserting $g = \theta^{(k)}$ and $l = \tfrac{7-2k}{2}$, for $0\leq k\leq 2$, and using Equations \eqref{eq:ThetaLogFirstDerivative} and \eqref{eq:ThetaLogSecondDerivative} once more, we see that Item (2) is implied by $\max\{(\bullet)^5,(\bullet)^{-5}\}\cdot \mathcal F(\psi)  \in \bd L^2(\RR^+)$. Altogether, $\theta = \mathcal F(\wlet)$ satisfies the conditions of \cite[Corollary 6.9]{bahowi15}, as desired, such that we can invoke said corollary. 
  
  Note that \cite[Corollary 6.9]{bahowi15} requires a pairwise association of the points $\lambda_{k,l}\in \Gamma$ and the elements $V^\delta_{k,l}$ of $\mathcal V^\delta$, for all $k,l\in\ZZ$, but the choice of point in $V^\delta_{k,l}$ is arbitrary. The covering $\mathcal V^\delta$ is, in fact, a tiling of $\Lambda$, such that $\lambda\in V^\delta_{k,l}$ implies $\lambda\notin V^\delta_{k',l'}$, for any $k,k',l,l'\in \ZZ$ with $\begin{psmallmatrix} k \\ l\end{psmallmatrix}\neq \begin{psmallmatrix} k' \\ l'\end{psmallmatrix}$. Hence, we can find $\Gamma_0 \subset\Gamma$, such that $|\Gamma_0 \cap V^\delta_{k,l}|=1$, for all $k,l\in\ZZ$, and $\Gamma_0$ is a set of stable sampling for $\Wlet$ by \cite[Corollary 6.9]{bahowi15}. 
  
  We can further find $\Gamma_1,\ldots,\Gamma_{N-1}\subset \Gamma$ that satisfy the following:
  \begin{enumerate}
    \item[1')] $\bigcup_{n=0}^{N-1} \Gamma_n = \Gamma$, and $\Gamma_n\cap \Gamma_{n'} =\emptyset$, for all $n,n'\in \{0,\ldots,N-1\}$ with $n\neq n'$. 
    \item[2')] $|\Gamma_n \cap V^\delta_{k,l}|\leq 1$, for all $n\in\{0,\ldots,N-1\}$ and all $k,l\in\ZZ$. 
  \end{enumerate}
  Since the points $(\lambda_{k,l})_{k,l\in\ZZ}$ with $\lambda_{k,l}\in V^\delta_{k,l}$, for all $k,l\in\ZZ$ in \cite[Corollary 6.9]{bahowi15} are arbitrary, it is clear that each $\Gamma_n$ satisfies the implicit 
  upper frame bound in that result. Hence, we conclude that $\Gamma$ is a set of stable sampling for $\Wlet$. Precisely, the lower bound in \eqref{eq:frameineq} equals the lower bound implied in \cite[Corollary 6.9]{bahowi15}, whereas the upper bound in \eqref{eq:frameineq} is no larger than $N$ times the upper bound implied in \cite[Corollary 6.9]{bahowi15}.
\end{proof}

\bibliography{addbib.bib}
\bibliographystyle{IEEEtran}

\end{document}